\documentclass[12pt]{amsart}

\usepackage{amsmath,amsfonts,bbm,amsthm}
\usepackage[latin1]{inputenc}  
\usepackage[T1]{fontenc}       
\usepackage[psamsfonts]{amssymb}
\usepackage{hyperref}

\renewcommand{\theequation}{\thesection.\arabic{equation}}
\newtheorem{theorem}{Theorem}
\newtheorem{lemma}{Lemma}
\newtheorem{proposition}{Proposition}
\newtheorem{corollary}{Corollary}
\newtheorem{remark}{Remark}
\newtheorem{definition}{Definition}

\newcommand{\eqnsection}{
\renewcommand{\theequation}{\thesection.\arabic{equation}}
    \makeatletter
    \csname  @addtoreset\endcsname{equation}{section}
    \makeatother}
\eqnsection

\def\z{{\mathbb Z}}

\def\ali{\hfill\break}

\def\W{\Omega}

\def\Z{{{\Bbb Z}}}
\def\P{{{\Bbb P}}}

\def\R{{{\Bbb R}}}
\def\C{{{\Bbb C}}}
\def\BR{{{\Bbb R}}}
\def\BC{{{\Bbb C}}}

\def\ue{{\underline e}}
\def\oe{{\overline e}}
\def\E{{{\Bbb E}}}
\def\dive{{\hbox{div}}}

\def\hhh{{{\mathcal H}}}

\def\indic{{{\mathbbm 1}}}
\def\Id{{\hbox{Id}}}

\def\cU{{\vert U\vert}}
\def\cE{{\vert E\vert}}

\def\oa{{\overline \alpha}}

\def\strength{{\hbox{strength}}}

\author[C. Sabot]{ Christophe Sabot}
\address{
Université de Lyon, Université Lyon 1, CNRS UMR5208, Institut
Camille Jordan, 43 bd du 11 nov., 69622 Villeurbanne cedex}
\email{sabot@math.univ-lyon1.fr}

\title[Dirichlet Environment ]{Random walks in random Dirichlet environment are transient in dimension $d\ge 3$ }

\keywords{Random walk in random environment, Dirichlet
distribution, Reinforced random walks, Max-Flow Min-Cut theorem}
\subjclass[2000]{primary 60K37, 60K35, secondary 5C20}
\thanks{This work was partly supported by the ANR project MEMEMO}

\begin{document}

\maketitle

{\bf Abstract:} We consider random walks in random Dirichlet
environment (RWDE) which is a special type of random walks in
random environment where the exit probabilities at each site are
i.i.d. Dirichlet random variables. On $\Z^d$, RWDE are
parameterized by a $2d$-uplet of positive reals. We prove that for
all values of the parameters, RWDE are transient in dimension
$d\ge 3$. We also prove that the Green function has some finite
moments and we characterize the finite moments. Our result is more
general and applies for example to finitely generated symmetric
transient Cayley graphs. In terms of reinforced random walks it
implies that directed edge reinforced random walks are transient
for $d\ge 3$.

\section{Introduction}

 Random Walks in Random Environment (RWRE) have
received a considerable attention in the last years. A lot is
known for one-dimensional RWRE, but the situation is far from
being so clear in dimension 2 and larger. Progress has been made
for multidimensional RWRE (in particular, since the work of
Sznitman and Zerner, cf \cite{SznitmanZ}) essentially in two
directions: for ballistic RWRE (cf \cite{SznitmanZ},
\cite{SznitmanCLT}, \cite{Sabot-Ann} and reference therein) and
more recently for small perturbations of the simple random walk in
dimension $d\ge 3$. Nevertheless, many important questions as the
characterization of recurrence, of ballistic behavior, invariance
principle remain open. Recently, for RWRE which are small
isotropic perturbations of the simple random walk on $\Z^d$, $d\ge
3$, Bolthausen and Zeitouni (\cite{BZ}) made some progress in the
direction of the invariance principle and proved transience using
renormalization techniques (invariance principle for the
corresponding continuous model has previously been obtained by
Sznitman and Zeitouni, \cite{SZ}). In the case of ballistic RWRE,
Rassoul-Agha and Seppäläinen obtained a quenched functional cental
limit theorem (\cite{RA-S}). On very special cases as symmetric
environments (\cite{Lawler}), or environments admitting a bounded
cycle representation (\cite{Deuschel}) an invariance principal has
been obtained. On the ballistic case, some important progress has
been made on the description of large deviations
(\cite{Varadhan,Yilmaz1,Yilmaz2,RA-S2}). We refer to
\cite{Zeitouni} or \cite{sz-review} for surveys.

Among random walks in random environment, random walks in random
Dirichlet environment (RWDE) play a special role. It corresponds
to the case where the transition probabilities at each site are
chosen as i.i.d. Dirichlet random variables. RWDE have the
remarkable property that the annealed law is the law of a directed
edge reinforced random walk. This is a simple consequence of the
representation of Polya's urns as a Dirichlet mixtures of i.i.d.
sampling (\cite{Pemantle1}, \cite{ES1}). In \cite{ES2}, N.
Enriquez and the author have obtained a criterion for ballistic
behavior on $\Z^d$, later improved by L. Tournier
(\cite{Tournier}). Besides, in the one-dimensional case, it
appeared in \cite{ESZ1} that limit theorems in the sub-ballistic
regime are fully explicit in the case of Dirichlet environment,
highlighting the special role of Dirichlet environment among
random environments. Finally, in \cite{hypergeom} we have
described a precise relation between RWDE and hypergeometric
integrals associated with certain arrangement of hyperplanes.

As already mentioned, random Dirichlet environment is very natural
from the point of view of reinforced random walks. Reinforced
random walks have been introduced by Diaconis (cf \cite{Diaconis}
and Pemantle's thesis \cite{Pemantle1}). There are three natural
models of linearly reinforced random walks, namely vertex
reinforced random walks, (undirected) edge reinforced random
walks, and directed edge reinforced random walks. None of these
models is completely understood yet. Vertex reinforced random
walks are expected to get localized on a finite set, but this is
completely solved only in dimension 1 (\cite{Volkov, Tarres}).
Concerning edge reinforced random walks, Diaconis and Coppersmith
proved that it can be represented as a "complicated" mixture of
reversible random walks. Recurrence of edge reinforced random
walks has been proved by Merkl and Rolles (\cite{Merkl1}) on a two
dimensional graph (but the question is still open on $\Z^2$
itself). By comparison, directed edge reinforced random walks
correspond to the annealed law of RWDE, hence they are a simple
mixture of non-reversible Markov chains. The difficulty of this
model does not come from the representation as a RWRE, but lies in
the non-reversibility of this RWRE. We refer to
\cite{Pemantle-survey} for a survey on the subject.

On $\Z^d$, RWDE are parameterized by $2d$ positive reals
$(\alpha_1, \ldots, \alpha_{2d})$, one for each direction. We call
$(\alpha_i)$ the weights. In this paper, we prove that RWDE on
$\Z^d$, for $d\ge 3$, are transient for all values of the weights
(in particular for the case of unbiased weights). In fact, we
prove that the Green function has some finite moments and we
explicitly compute the critical integrability exponent $\kappa$,
i.e. the supremum of the reals $s>0$ such that the Green function
$G(0,0)^s$ has finite expectation. We show that this real is the
same as the one for the RWDE killed when it exists a finite large
ball (which has been computed by Tournier in \cite{Tournier}). In
some way, it means that trapping is only due to finite size traps
which come from the non uniform ellipticity of the environment.
Our result is in fact more general and applies for example to any
finitely generated Cayley graph on which the simple random walk is
transient (cf theorem \ref{t-oriented} and corollary
\ref{Cayley}). Compared to the results of \cite{BZ}, our results
are non-perturbative, they apply to any choice of the weights, but
our method is specific to the Dirichlet environment. The proof is
remarkably short, and it is quite surprising that for RWDE things
simplify so much. We don't have a clear explanation for this, it
may be related to the correspondence with linearly reinforced
random walks (even if it does not appear in this proof) or from
the relation between RWDE and hypergeometric integrals (cf
\cite{hypergeom}). We think it highlights the special role of
Dirichlet environment and we think that the techniques presented
in this paper will help us to understand more on RWDE.

Our proof is based on an explicit formula valid for Dirichlet
environments and on a method of perturbation of the weights.
Finally, the critical exponent is obtained thanks to an $L_2$
version of the Max-Flow Min-Cut theorem (proposition
\ref{maxflowL2}). The explicit formula (corollary \ref{W}) was in
fact hinted in our joint work with N. Enriquez and O. Zindy
(\cite{ESZ1, ESZ2}) where it appeared that in the case of
one-dimensional RWDE, the Green function on the half-line at 0 has
an explicit law (which is in fact a consequence of a result of
Chamayou and Letac on explicit solutions of renewal equations,
\cite{Letac}).
 \ali

Let us present our results for the graph $\Z^d$. Let $(e_1, \ldots
,e_d)$ be the canonical base of $\Z^d$, and set $e_j=-e_{j-d}$,
for $j=d+1, \ldots ,2d$. The set $\{e_1, \ldots ,e_{2d}\}$ is the
set of unit vectors of $\Z^d$. We consider a probability law
$\lambda$ on
\begin{eqnarray}
\label{T2d}
 \{ (x_1, \ldots ,x_{2d})\in ]0,1]^{2d},\;\;\; \sum_{i=1}^{2d} x_i=1\}
 \end{eqnarray}
 We construct a Markov chain on $\Z^d$, to nearest neighbors as
 follows: we choose independently at
 each point $x\in \Z^d$, some transition probabilities
 $$
 p(x)=\left( p(x,x+e_i)\right)_{i=1,\ldots ,2d},
 $$
 according to the law $\lambda$. It means that the vector
 $( p(x))_{x\in \Z^d}$ is chosen according to the product measure
 $\mu=\otimes_{x\in \Z^d} \lambda$. It defines the transition
 probability of a Markov chain on $\Z^d$, and we denote by $P^{(p)}_x$
 the law of this Markov chain starting from $x$:
 $$
 P^{(p)}_x[ X_{n+1}=x+e_i| X_n=x]=p(x,x+e_i).
 $$

Random Dirichlet environments correspond to the case where the law
$\lambda$ is a Dirichlet law (cf section 2 for details).
More precisely, we choose some
positive weights $(\alpha_1, \ldots ,\alpha_{2d})$ and we take
$\lambda=\lambda^{(\alpha)}$ with density
\begin{eqnarray} \label{simplex}
{\Gamma(\sum_{i=1}^{2d} \alpha_i)\over  \prod_{i=1}^{2d}
\Gamma(\alpha_i)} \left( \prod_{i=1}^{2d} x_i^{\alpha_i-1} \right)
dx_1\cdots dx_{2d-1},
\end{eqnarray}
where $\Gamma$ is the usual Gamma function
$\Gamma(\alpha)=\int_0^\infty t^{\alpha -1} e^{-t} dt$. (In the
previous expression $dx_1\cdots dx_{2d-1}$ represents the image of
the Lebesgue measure on $\R^{2d-1}$ by the application $(x_1,
\ldots, x_{2d-1})\rightarrow (x_1, \ldots, x_{2d-1},
1-(x_1+\cdots+x_{2d-1})$. Obviously, the law does not depends on
the specific role of $x_{2d}$.) We denote by $\P^{(\alpha)}$ the
law obtained on the environment in this way. This type of
environment plays a specific role, since the annealed law
$\P^{(\alpha)}_x[\cdot]=\E^{(\alpha)}[ P^{(p)}_x(\cdot)]$
corresponds to a directed edge reinforced random walk with an
affine reinforcement, i.e.
$$
\P^{(\alpha)}_x [X_{n+1}=X_n+e_i | \sigma(X_k, k\le n)]=
{\alpha_i+N_i(X_n, n)\over \sum_{k=1}^{2d} \alpha_k+N_k(X_n, n)},
$$
where $N_k(x,n)$ is the number of crossings of the directed edge
$(x,x+e_k)$ up to time $n$ (cf \cite{ES1}). When the weights are
constant equal to $\alpha$, the environment is isotropic: when
$\alpha$ is large, the environment is close to the deterministic
environment of the simple random walk, when $\alpha$ is small the
environment is very disordered.

Let us now describe precisely our results for  $\Z^d$, $d\ge 3$.
We denote by $G(x,y)$ the green function in the environment
$(p(x))_{x\in \Z^d}$.
$$
G(x,y)=E_x^{(p)}[\sum_{k=0}^\infty \indic_{X_k=y}].
$$
\begin{theorem}\label{Zd}
For $d\ge 3$ and for any choice of weights $(\alpha_1, \ldots,
\alpha_{2d})$ we have
$$
\E^{(\alpha)}\left( G(0,0)^s\right)<\infty
$$
if and only if $s<\kappa$ where
$$
\kappa=2\sum_{j=1}^{2d} \alpha_{e_j} -\max_{i=1, \ldots,
d}(\alpha_{e_i}+\alpha_{-e_i}).
$$
In particular, the RWDE is transient for almost all environments.
\end{theorem}
\begin{remark} In \cite{Tournier}, Tournier computed the critical
integrability exponent for RWDE on finite graphs. For $N>1$, let
$G_N$ be the Green function of the RWDE killed when it exists the
ball $B(0,N)$: theorem 2 of \cite{Tournier} implies that
$G_N(0,0)^s$ is integrable if and only if $s<\kappa$. Hence,
$G_N(0,0)$ has no higher integrable moments than the Green
function $G(0,0)$ itself. It seems to mean that there is no
infinite size trap and that the trapping effect comes only from
finite size traps due to the non-uniform ellipticity of the
environment (cf also remark \ref{trapps} after theorem
\ref{t-oriented}).
\end{remark}
\begin{remark}
The result for $\Z^d$ is in fact a consequence of a more general
result valid for directed symmetric graph, under a condition on
the weights (positive divergence), cf theorem \ref{t-oriented}.
\end{remark}
\begin{remark}
This theorem solves, in the special case of Dirichlet
environments, problem 3 stated in Kalikow's paper \cite{Kalikow} :
Kalikow's problem is to prove that all RWRE with i.i.d. elliptic
environments in dimension $d\ge 3$ are transient.
\end{remark}
\begin{remark}
Let us make some comments on previous results on recurrence or
transience of RWRE. In \cite{Lawler}, Lawler proved a CLT for RWRE
with symmetric environments (i.e. almost surely, the environment
is symmetric at each site in each direction) and recurrence on $\Z^2$ and
transience on $\Z^d$, $d\ge 3$, are proved in \cite{Zeitouni}, theorem 3.3.22.
Transience has been proved
by Bolthausen and Zeitouni (\cite{BZ}) in dimension $d\ge 3$ for
small isotropic random perturbations of the simple random walk (in
fact, there is no intersection between the cases treated in
\cite{BZ} and in this paper). In the same paper they also obtain
some estimates on the exit distributions of large balls. In
\cite{SZ}, Sznitman and Zeitouni proved an invariance principal
for diffusions in isotropic environment at low disorder in
dimension $d\ge 3$ (which is the continuous analogue of isotropic
RWRE at low disorder investigated in \cite{Bricmond-Kupianen}) and
obtained transience as a by-product.

In the closely related model of (undirected) edge reinforced
random walks they are very few available results. Recurrence has
been proved on a graph of the type of $\Z^2$ by Merkl and Rolles
(\cite{Merkl1}). On regular trees transience or recurrence depends
on the reinforcement parameter (cf \cite{Pemantle}). On $\Z^d$,
$d\ge 3$, there is, up to my knowledge, no clear conjecture.
\end{remark}
\begin{remark}
The exponent $\kappa$ should play an important role in the
asymptotic behaviour of the RWDE. Indeed, $\kappa$ is related to
the tail of the expected number of visits to the point 0. In
dimension 1 in the transient case, the exponent $\kappa$ defined
as the supremum of the $s>0$ such that $G(0,0)^s$ is integrable is
also the exponent which governs the asymptotic behaviour of the
walk (cf \cite{KKS}). For RWDE, in dimension $d\ge 3$, with
non-balanced weights (i.e. $\alpha_{i+d}\neq \alpha_i$ for a
direction $i=1,\ldots ,d$) this result leads to conjecture that
the RWDE is ballistic if and only if $\kappa>1$. It has been
proved in \cite{Tournier}, proposition 11, that the RWDE has null
velocity when $\kappa\le 1$.
\end{remark}

Let us describe the organization of the paper. In section 2, we
describe the model of RWDE on directed graphs and in section 3 we
state the general results on symmetric graphs. In section 4, we
prove the key explicit formulas. Section 5 is devoted to the proof
of the main result on integrability and section 6 to the
transience part of the result. In section 7, we prove a
generalization of the Max-Flow Min-Cut theorem for flows of finite
energy. In section 8 we apply the results to directed symmetric
graphs and we prove theorem \ref{t-oriented} (ii). Remark that the
proof of transience on $\Z^d$, $d\ge 3$, can be understood by
reading sections 2,4,5 only.

\section{Markov chain in  Dirichlet environment on directed graphs}

The Dirichlet law is the multivariate generalization
of the beta law. The Dirichlet law with parameters
$(\alpha_1, \ldots, \alpha_N)$, $\alpha_i>0$ is the law on the
simplex
$$
\{(p_1, \ldots, p_N)\in [0,1]^N,\; \hbox{ such that $\sum
p_i=1$}\},
$$
with distribution
$$
\left({\Gamma(\sum_{i=1}^N \alpha_i)\over\prod_{i=1}^N
\Gamma(\alpha_i)} \prod_{i=1}^N p_i^{\alpha_i-1}\right) dx_1\cdots
dx_{N-1},
$$
where as previously $dx_1\cdots dx_{N-1}$ represents the image of
the Lebesgue measure on $\R^{N-1}$ by the application $(x_1,
\ldots, x_{N-1})\rightarrow (x_1, \ldots, x_{N-1},
1-(x_1+\cdots+x_{N-1}))$ (which does not depend on the specific
role of $x_N$). The first coordinate of a Dirichlet random
variable with parameters $(\alpha, \beta)$ is by definition a beta
random variable with parameter $(\alpha, \beta)$. The following representation
of Dirichlet ditribution is classical (cf e.g. \cite{Wilks}, page 180): if $(\gamma_1, \ldots, \gamma_N)$ are independent
gamma random variables with parameters $\alpha_1, \ldots, \alpha_N$ then
$({\gamma_1\over \sum \gamma_i}, \ldots,{\gamma_N\over \sum \gamma_i})$ is a
Dirichlet random variable with parameters $(\alpha_1, \ldots, \alpha_N)$.
The following properties are easy consequences of this representation (cf \cite{Wilks} page 179-182).

\noindent {\bf (Associativity)} Let $I_1, \ldots, I_k$ be a
partition of $\{1, \ldots N\}$. Then the random variable
$(\sum_{j\in I_i} p_j)_{i=1, \ldots, k}$ is a Dirichlet random
variable with parameters $(\sum_{j\in I_i} \alpha_j)_{i=1, \ldots,
k}$.

\noindent {\bf (Restriction)} Let $J$ be a non-empty subset of
$\{1, \ldots, N\}$. The random variable $({p_j\over \sum_{i\in J}
p_i})_{j\in J}$ follows a Dirichlet law with parameters
$(\alpha_j)_{j\in J}$ and is independent of $\sum_{i\in J} p_i$
(which follows a beta random variable with parameters $(\sum_J
\alpha_i, \sum_{J^c} \alpha_i)$ by the associativity property.)

Let us first describe Random Walks  in Dirichlet Environment (RWDE
for short) on a general graph. A directed graph is a couple
$G=(V,E)$ where $V$ is the countable set of vertices and $E$ is
the countable set of edges. By definition, to an edge $e$
corresponds a couple of vertices $(\underline e, \overline e)$ :
$\underline e$ and $\overline e$ represent respectively the tail
and the head of the edge $e$.  For convenience we allow multiple
edges and loops (i.e. edges with $\underline e=\overline e)$. We
suppose that the graph has bounded degree i.e. that the number of
edges exiting a vertex $x$ or pointing to a vertex $x$ is bounded.
For an integer $n$ and a vertex $x$ we denote by $B(x,n)$ the ball
with center $x$ and radius $n$ for the graph distance (defined
independently of the orientation of the edges). We say that a
subset $K\subset V$ is strongly connected if for any two vertices
$x$ and $y$ in $K$ there is a directed path in $K$ from $x$ to
$y$.

 Let us define the divergence operator on the graph $G$: it is the
 operator $\dive:\R^E\mapsto \R^V$ defined for a function
 $\theta:E\mapsto \R$ by
 $$
 \dive(\theta)(x)= \sum_{e, \ue=x} \theta(e)-\sum_{e,\oe=x}
\theta(e), \;\;\; \forall x\in V .
$$

The set of environments  on $G$ is defined as the set
$$
\Delta=\{(p_e)_{e\in E}\in ]0,1]^E,\;  \hbox{ such that } \sum_{e,
\ue=x} p_e=1, \; \forall x\in V\hbox{ such that $\{e,\;
\ue=x\}\neq \emptyset$}\}.
$$
With any environment $(p_e)$ we associate the Markov chain on $V$
with law $(P_x^{(p)})$, where $P^{(p)}_{x_0}$ is the law of the
Markov chain starting from $x_0$ with transition probabilities
given by
$$
P^{(p)}_{x_0}(X_{n+1}=y|X_n=x)=p_{(x,y)}=\sum_{e, \ue=x, \oe=y}
p_e, \;\;\; \forall x\neq y.
$$
If $y$ is a vertex such that no edge is exiting $y$ then we put $
P^{(p)}_{x_0} (X_{n+1}=y|X_n=y)=1 $, so that $y$ is an absorbing
point. We denote by $G^{(p)}(x,y)$ the Green function in the
environment $(p)$ defined for $x$ and $y$ in $V$ by
$$
G^{(p)}(x,y)=E^{(p)}_x\left(\sum_{k=0}^\infty
\indic_{X_k=y}\right).
$$
(We often simply write $G(x,y)$ for $G^{(p)}(x,y)$).

Let $(\alpha_e)_{e\in E}$ be a set of positive weights on the
edges. For any vertex $x$ we set
$$
\alpha_x=\sum_{e, \ue=x}\alpha_e
$$
the sum of the weights of the edges with origin $x$.
The Dirichlet environment with parameters $(\alpha_e)$, denoted by
$\P^{(\alpha)}$, is the law on $\Delta$ obtained by taking at each
site $x$ the transition probabilities $(p_e)_{\ue=x}$
independently accordingly to the Dirichlet law with parameters
$(\alpha_e)_{\ue=x}$.


When the graph is finite the distribution of the Dirichlet
environment is given by
\begin{eqnarray}
\label{measure-Delta}
 {\prod_{x\in V} \Gamma(\alpha_x)\over {\prod_{e\in
E} \Gamma(\alpha_e)}} \left( \prod_{e\in E} p_e^{\alpha_e
-1}\right) d\lambda_{\Delta}.
\end{eqnarray}
where $d\lambda_{\Delta}$ is the measure on $\Delta$ given by
\begin{eqnarray*}
d\lambda_{\Delta} =\prod_{e\in \tilde E} dp_e,
\end{eqnarray*}
where $\tilde E$ is obtained from $E$ by removing arbitrarily, for
each vertex $x$, one edge with origin $x$ (easily, it is
independent of this choice)

\begin{remark}
We allow multiple edges for convenience (for the proof of
corollary \ref{W}) but it does not play any role in terms of the
process: indeed, if $G=(V,E)$ with weights $(\alpha_e)_{e\in E}$
has multiple edges we can consider the quotiented graph $\tilde
G=(V, \tilde E)$ with weights $(\tilde\alpha_e)$ obtained by
replacing multiple edges by a unique edge with weight equal to the
sum of the weights of the corresponding edges in $G$. If
$(p_e)_{e\in E}$ is a Dirichlet environment on the graph $G$, the
corresponding environment on the quotiented graph $\tilde G$
obtained by summing the $(p_e)$ of multiple edges is again a
Dirichlet environment on $\tilde G$ with weights
$(\tilde\alpha_e)$. This is due to the associativity property of
Dirichlet laws.
\end{remark}

\section{ Statement of the results on transient symmetric graphs}

For us, a directed symmetric graph will be a directed graph $G$
without multiple edges and such that if $(x,y)\in E$ then
$(y,x)\in E$. To $G$ corresponds a non-directed graph $\overline
G=(V, \overline E)$ where $\{x,y\}\in \overline E$ if and only if
$(x,y)\in E$.

\begin{theorem}\label{t-oriented}
Let $G$ be a connected directed symmetric graph with bounded
degree. Suppose that the weights $(\alpha_e)_{e\in E}$ satisfy

 \indent (H1) there exists $c>0$ and $C>c$ such that $c\le
 \alpha_e\le C$ for all $e$ in $E$.

\indent (H2) For all vertices $x$,  $\dive(\alpha)(x)\ge 0$.

Suppose that the simple random walk on the non-directed graph
$\overline G$ is transient.

\noindent (i) For any $x_0$ in $V$, there exists $\kappa_0>0$ such
that
$$
\E^{(\alpha)}(G(x_0,x_0)^s)<\infty
$$
for all $s<\kappa_0$. In particular, the RWDE on $G$ with
parameter $(\alpha_e)$ is transient for almost all environments.

\noindent (ii) Assume moreover that the following condition on $G$
holds

(H'3) There exists a strictly increasing sequence of integers
$\eta_n$ such that $B(x_0, \eta_{n+1})\setminus B(x_0, \eta_n)$ is
connected in $G$.

Then

$$
\E^{(\alpha)}(G(x_0,x_0)^s)<\infty,\;\hbox{ if and only if
$s<\kappa$}
$$
where if $(x_0,x_0)\not\in E$
\begin{eqnarray*}
\kappa=\min\{\alpha(\partial_E K), \; \hbox{ $K\subset V$ is
finite connected in $\overline G$, $x_0\in K$ and
$K\neq\{x_0\}$,}\}
\end{eqnarray*}
and if $(x_0,x_0)\in E$
\begin{eqnarray*}
\kappa=\min\{\alpha(\partial_E K), \; \hbox{ $K\subset V$ is
finite connected in $\overline G$, $x_0\in K$.}\}
\end{eqnarray*}
with $\partial_E K=\{e\in E,\; \ue\in K, \oe\notin K\}$ and
$\alpha(\partial_E K)=\sum_{e\in \partial_E K} \alpha_e$.
\end{theorem}
\begin{remark} An expression for $\kappa_0$ in terms of a
$L_2$-Max-Flow problem is given in the proof, cf formula
\ref{kappa0}.
\end{remark}
\begin{remark}\label{trapps}
The interpretation of $\kappa$ is the following : the finite
subsets $K$ which appear in the infimum should be understood as
finite traps and the value $\alpha(\partial_E K)$ represents the
strength of the trap $K$ : indeed, $\alpha(\partial_E K)$ governs
the tail of the probability that the values of $(p_e)_{e\in
\partial_E(K)}$ are all smaller than $\epsilon$ ; when the exit probabilities $(p_e)_{e\in \partial_E K}$
are small the process spends a long time in the strongly connected
region $K$. In \cite{Tournier}, corollary 4, Tournier computed the
critical integrability exponent for RWDE on finite graphs: (ii)
shows that $\kappa$ is the same as the critical integrability
exponent of the Green function of the RWDE killed when it exits
the ball $B(x_0,N)$, for $N$ such that $B(x_0,N)$ contains the
subset $K$ which realizes the minimum in the expression of
$\kappa$. This suggests that trapping only comes from finite size
traps on transient symmetric graphs that satisfy (H'3).
\end{remark}
\begin{remark}
The difference between the expression of $\kappa$ when there is a
loop at $x_0$ or not comes from the fact that when $(x_0,x_0)\in
E$ the RWDE can be trapped on $\{x_0\}$ but not when
$(x_0,x_0)\not\in E$.
\end{remark}

\begin{corollary}\label{Cayley}
Let $\{e_1, \ldots, e_d\}$ be a finite, symmetric (i.e. the set is
stable by inversion), set of generators of a group and $G$ be the
associated Cayley graph. Let $(\alpha_1, \ldots, \alpha_d)$ be
positive reals. If the simple random walk on the Caley graph is
transient, then the RWDE on the Caley graph $G$, with weights
$(\alpha_{(g,g e_i)}=\alpha_i)$, is transient almost surely.
\end{corollary}
\begin{proof}
(H1) is clearly true. For all element of the group $g$,
$\sum_{i=1}^N \alpha_{g,g e_i}= \sum_{i=1}^N \alpha_i=\sum_{i=1}^N
\alpha_{g e_i^{-1}, g}$. Thus (H2) is satisfied.
\end{proof}

 \section{Stability by time reversal}

\subsection{The key lemma}

Suppose now that the graph is finite and strongly connected i.e.
that there is a directed path between any two vertices $x$ and
$y$. Let $\check G=(V, \check E)$ be the graph obtained from $G$
by reversing all the edges, i.e. $\check E$ is obtained from $E$
by reversing the head and the tail of the edges. If $e\in E$ we
denote by $\check e\in \check E$ the reversed edge with tail $\oe$
and head $\ue$. For an environment $(p_e)$, $(\check p_{\check
e})$ denotes the environment obtained by time reversal of the
Markov chain, i.e. for all $e$ in $E$
$$
\check p_{\check e}={\pi_\ue\over \pi_\oe}  p_{e},
$$
where $(\pi_x)_{x\in V}$ is the invariant probability on $V$ for
the Markov chain $P^{(p)}$. (Since we assume that the graph is
strongly connected and that the weights $p_e$ are positive this
invariant probability exists and is unique).
\begin{lemma}\label{reversal}
Let $G=(V,E)$ be a strongly connected finite directed graph.
Suppose that the weights $(\alpha_e)$ have divergence null, i.e.
$$
\dive(\alpha)(x)=0, \;\;\; \forall x\in V,
$$
If $(p_e)_{e\in E}$ is a Dirichlet environment with parameters
$(\alpha_e)_{e\in E}$ then the time reversed environment $(\check
p_e)_{e\in \check E}$ is a Dirichlet environment on $\check G$
with parameters $(\alpha_e)$.
\end{lemma}
\begin{remark} By this we mean that $(\check p_e)$ is distributed
according to a Dirichlet environment with parameters $(\check
\alpha_{\check e\in \check E})$ where $\check \alpha_{\check
e}=\alpha_e$ if $\check e$ is the reversed edge of $e$. We will
often identify the edges in $E$ with their reversed edges in
$\check E$.
\end{remark}
\begin{proof} Two proofs are available for this lemma. The
original proof given in this paper is analytic and based on a
change of variable. Later, in collaboration with L. Tournier, we
obtained a shorter probabilistic proof, cf \cite{Sabot-Tournier}.
The analytic proof has nevertheless an interest since it leads to
an interesting distribution on the space of occupation density (cf
lemma  and remark below).

Let $e_0$ be a specified edge of the graph. Let $\hhh_{e_0}$ be
the affine space defined by
$$
\hhh_{e_0}=\{ (z_e)_{e\in E}\in \R^E, \;\; z_{e_0}=1, \;\;
\dive(z)\equiv 0\}.
$$
and $\hhh$ the vector space
$$
\hhh=\{ (z_e)_{e\in E}\in \R^E, \;\;  \dive(z)\equiv 0\}.
$$
Let $\tilde \Delta_{e_0}=\hhh_{e_0}\cap(\R_+^*)^E$.
 The strategy is to make the change of
variable
\begin{eqnarray*}
\Delta&\mapsto &\tilde\Delta_{e_0}
\\
(p_e)_{e\in E}&\to &(z_e={\pi_{\ue}p_e\over
\pi_{\ue_0}p_{e_0}})_{e\in E}.
\end{eqnarray*}
Hence, $(z_e)$ is the occupation time of the graph normalized so
that $z_{e_0}=1$. It is easy to see that the previous change of
variable is a $C^\infty$-diffeomorphism. Let $T$ be a spanning
tree of the graph $G$ such that $e_0\notin T$. (This is possible
since the graph is strongly connected and thus $e_0$ belongs to at
least one directed cycle of the graph.) We denote by
$B=T\cup\{e_0\}$. Then $(z_e)_{e\in T^c}$ is a base of $\hhh$ and
$(z_e)_{e\in B^c}$ is a base of $\hhh_{e_0}$. Let $x_0\in V$ be
any vertex, and set $U=V\setminus\{x_0\}$. We need the following
lemma.
\begin{lemma}\label{chgt-variables}
Let $\psi$ be a positive test function on $\Delta$. Then
$$
\int_\Delta \Psi((p_e)) \left( \prod_{e\in
E}p_e^{\alpha_e-1}\right)  d\lambda_\Delta= \int_{\tilde
\Delta_{e_0}} \Psi(({z_e\over z_{\ue}})) \left( {\prod_{e\in E}
z_e^{\alpha_e-1} \over \prod_{x\in V} z_x^{\alpha_x}}\right)
\det\left( Z_{|U\times U}\right) \prod_{e\in B^c} dz_e.
$$
where as usual $z_x=\sum_{e,\ue=x} z_e$ and $Z$ is the $V\times V$
matrix defined by
$$
Z_{x,x}= z_x, \; \forall x\in V, \;\;\;
Z_{x,y}=-z_{x,y}=-\sum_{e,\ue=x,\oe=y} z_e, \;\; \forall x\neq y.
$$
\end{lemma}
\begin{remark}  This formula is essentially the same as the one which
gives the correspondence between RWDE and hypergeometric integrals
associated with an arrangement of hyperplanes in \cite{hypergeom}.
\end{remark}
\begin{remark}
This lemma expresses the law of the random environment in the
variables $(z_e)$ which correspond to the occupation densities of
the edges (properly renormalized). We can remark that this formula
is reminiscent of the distribution discovered by Diaconis and
Coppersmith
(\cite{Coppersmith-Diaconis,Diaconis,Keane-Rolles,Diaconis-Rolles})
which expresses edge-reinforced random walk as a mixture of
reversible Markov chains.
\end{remark}

We see that lemma \ref{reversal} is a direct consequence of the
previous result. Indeed we see that lemma \ref{chgt-variables}
applied to the reversed graph $(\check G, \check E)$, starting
with the weights $\check\alpha_{\check e}=\alpha_e$ gives the same
integrand with $\alpha_x$ replaced by $\oa_x=\sum_{e,
\oe=x}\alpha_e$. The two coincide after the change of variables
exactly when $\dive(\alpha)\equiv 0$.
\end{proof}

The proof of lemma \ref{chgt-variables} needs some lengthy
computation and is deferred to the appendix in section 8.

\subsection{Applications}
Consider now a finite graph with a cemetery point, i.e. we suppose
that $G=(V,E)$, that $V$ and $E$ are  finite and that $V$ can be
written $V=U\cup\{\delta\}$ where
 \begin{itemize}
\item no edge is exiting $\delta$ ($\delta$ is called the cemetery
point)
 \item for any point $x$ in $U$ there is a directed path from $x$ to
$\delta$.
 \end{itemize}
It means that $\delta$ is absorbing for the Markov chain on $G$
with law $P^{(p)}$. For $x$ and $y$ in $U$ we denote by
$G^{(p)}(x,y)$ the Green function in the environment $(p_e)$
$$
G^{(p)}(x,y)=E^{(p)}_x(\sum_{k=0}^\infty \indic_{X_k=y}) .
$$
\begin{corollary}\label{W}

(i) Suppose that $\dive(\alpha)(x)=0$ for all $x$ in $U$ such that
$x\neq x_0$ then $\dive(\alpha)(x_0)>0$ and $G^{(p)}(x_0,x_0)$ is
distributed as ${1\over W}$ where $W$ is a beta random variable
with parameter $(\dive(\alpha)(x_0),
\alpha_{x_0}-\dive(\alpha)(x_0))$.

(ii) Suppose that $\dive(\alpha)(x)$ is non-negative for all $x$
in $U$. Let $x_0$ be a vertex in $U$ such that
$\dive(\alpha)(x_0)\ge \gamma>0$, then $G^{(p)}(x_0,x_0)$ is
stochastically dominated by ${1\over W}$ where $W$ is a beta
random variable with parameter $(\gamma, \alpha_{x_0}-\gamma)$.
\end{corollary}
\begin{remark} The explicit formula in (i) was in fact suggested by our
joint work with N. Enriquez and O. Zindy (\cite{ESZ1}) and also by
the correspondence established in \cite{hypergeom}. It appeared in
\cite{ESZ1} that in the case of sub-ballistic one-dimensional RWRE
limit theorems are fully explicit in the case of Dirichlet
environments. This is a consequence of the fact that the Green
function at 0 of the RWDE on the half line $\Z_+$ is equal in law
to $1/W$ where $W$ is a beta random variable with appropriate
weights. In the case of one-dimensional RWDE, this explicit
formula is a consequence of a result of Chamayou and Letac
(\cite{Letac}) on explicit solutions for renewal equations. From
the point of view of the correspondence with hypergeometric
integrals described in \cite{hypergeom}, the condition of null
divergence corresponds to a condition of resonance of the weights
of the arrangement. Resonant arrangements are not well understood
yet.

\end{remark}
 \begin{proof} (of corollary \ref{W})
(i) We can freely suppose that any $y$ in $U$ can be reached
following a directed path from $x_0$ (indeed, the part of the
graph which cannot be reached from $x_0$ does not play any role in
$G(x_0,x_0)$). Suppose that $\dive(\alpha)(x_0)=\gamma$,
$\gamma>0$. It means that $\dive(\alpha)(\delta)=-\gamma$.
Consider now the graph $\tilde G=(U,E)$ obtained by identification
of the vertices $\delta$ and $x_0$. The edges of $\tilde G$ are
just obtained by identification of $\delta$ and $x_0$ in the edges
of $G$ (with possibly creation of multiples edges and loops in
$\tilde G$), and we denote by $\tilde x_0$ the point corresponding to
the indentification of $\delta$ and $x_0$.
The graph $\tilde G$ is clearly strongly connected
and if we keep the same weights on the edges we have
$\dive^{\tilde G}(\alpha)\equiv 0$. Consider the invariant
probability $(\pi_x)_{x\in U}$ for the RWDE on $\tilde G$. It
gives the occupation time on the edges $z_e=\pi_{\ue} p_e$, and
the time reversal transition probabilities $\check p_e={z_e\over
\pi_{\oe}}$. If $G^{(p)}(x_0,x_0)$ is the Green function on the
graph $G$ (which is the graph with the cemetery point)
$$
G^{(p)}(x_0,x_0)={1\over \sum_{\oe =\delta} \check p_{\check e}}.
$$
(where $\oe=\delta$ is relative to the edges in the initial graph
$G$). Indeed, we have
$$
{1\over G^{(p)}(x_0,x_0)} = \sum_{\sigma \in \Sigma} p_\sigma,
$$
where $\Sigma$ is the set of direct paths $\sigma$ in $G$ from
$x_0$ to $\delta$ i.e. $\sigma=(e_0, \ldots , e_{n-1})$ with
$\overline e_{i-1}=\underline e_{i}$ for $i=1, \ldots , n-1$ and
$\underline e_0=x_0$, $\overline e_{n-1}=\delta$ and $\underline
e_i\not\in \{x_0, \delta\}$ for $i=1, \ldots , n-1$. Since in
$\tilde G$ $\sigma\in \Sigma$ is a cycle from $\tilde x_0$ to $\tilde x_0$, we
have
$$
{1\over G^{(p)}(x_0,x_0)} = \sum_{\sigma \in \Sigma} \check
p_{\check \sigma}.
$$
where $\check \sigma$ is the returned path $(\check e_{n-1},
\ldots, \check e_0)$ if $\sigma=(e_0, \ldots, e_{n-1})$. But since
there is no edge exiting $\delta$ in the graph $G$ we have
$$
\sum_{\sigma \in \Sigma} \check p_{\check \sigma}=\sum_{\oe
=\delta} \check p_{\check e}.
$$

By the previous lemma we know that $(\check p_e)_{\oe=\delta
\hbox{ or } \oe=x_0}$ has the law of  a Dirichlet random variable
with parameters $(\alpha_e)_{{\oe=\delta \hbox{ or } \oe=x_0}}$.
(Indeed, the vertices $\delta$ and $x_0$ are identified in the
quotiented graph $\tilde G$.) The conclusion comes from the
associativity property of Dirichlet random variables (cf section
2).

(ii) Consider the graph $\tilde G=(\tilde U\cup\{\tilde
\delta\},\tilde E)$ obtained as follows. The set of vertices is
defined by $\tilde U=U\cup\{\delta\}$ and $\tilde \delta$ is the
new cemetery point. The set of edges $\tilde E$ is obtained by
adding to the edges $E$ of $G$ some new edges with origin
$\delta$: for $x\neq x_0$ in $U$ such that $\dive(\alpha)(x)>0$ we
add the edge $(\delta, x)$ with weight $\dive(\alpha)(x)$; for
$x_0$, if $\dive(\alpha)(x_0)>\gamma$ we put a new edge
$(\delta,x_0)$ with weight $\dive(\alpha)(x_0)-\gamma$; we also
add the edge $(\delta, \tilde \delta)$ with weight $\gamma$.
Clearly, the new graph $\tilde G$ with the previous choice of
weights (denoted by $\tilde \alpha$) satisfies the condition of
(i) since $\dive(\tilde \alpha)_{|\tilde U}=\gamma \delta_{x_0}$.
Moreover if $(p_e)_{e\in E}$ is a Dirichlet environment on $G$
with weights $(\alpha_e)$ it can be extended to a Dirichlet
environment $(\tilde p_e)_{e\in \tilde E}$ on $\tilde G$ just by
choosing independently the transition probabilities on the edges
exiting $\delta$ according to a Dirichlet random variable with the
appropriate weights. Remark that
$$
G^{(p)}(x_0,x_0)\le G^{(\tilde p)}(x_0,x_0).
$$
Indeed, the Markov chains on $G$ and $\tilde G$ behave the same as
long as they are on $U$ but the Markov chain on $G$ is stuck on
$\delta$ although the Markov chain on $\tilde G$ can come back to
$U$ from $\delta$. The conclusion is a consequence of (i) since
(i) implies that $G^{(\tilde p)}(x_0,x_0)$ has the law of $1/W$
with $W$ a beta random variable with weights $(\gamma,
\alpha_{x_0}-\gamma)$.
\end{proof}

\section{Integrability condition}
Suppose now that $G=(V,E)$ is a countable connected directed graph
with bounded degree. Suppose for simplicity that there is at least
one edge exiting each vertex.

We recall that a flow from a vertex $x_0$ to infinity (cf
\cite{PL}) is a positive function $\theta$ on the edges such that
$$
\dive(\theta)(x)=0, \;\;\; \forall x\neq x_0,
$$
and
$$
\dive(\theta)(x_0)\ge 0.
$$
The strength of the flow is the value
$$
\strength(\theta)=\dive(\theta)(x_0).
$$
A unit flow is a flow of strength  1.

We say that the flow $\theta$ has finite energy if it is square
integrable i.e. if
$$
\sum_{e\in E} \theta_e^2 <\infty.
$$
\begin{theorem}
\label{integrability} Let $(\alpha_e)_{e\in E}$ be a family of
positive weights on the edges which satisfy

 \indent (H1) there exists $c>0$ and $C\ge c$ such that $c\le
 \alpha_e\le C$ for all $e$ in $E$.

\indent (H2) For all vertices $x$,  $\dive(\alpha)(x)\ge 0$.

Suppose that $\theta$ is a unit flow with finite energy from $x_0$
to infinity then
$$
\E^{(\alpha)}\left( G(x_0,x_0)^s\right)<\infty
$$
as soon as
$$
s <\inf_{e\in E} {\alpha_e\over \theta_e}
$$
\end{theorem}
\begin{proof}
Let $\gamma$ be a positive real. We define the weights
$$\alpha^\gamma=\alpha+\gamma \theta.
$$
 We clearly have
$$
\dive(\alpha^\gamma)\ge \gamma\delta_{x_0}.
$$
For a positive integer $N$, let $U_N$ be the ball with center
$x_0$ and radius $N$ in $G$. We define the graph
$G_N=(U_N\cup\{\delta\},E_N)$ as follows. We contract all the
vertices of $(U_N)^c$ to the cemetery point $\delta$. The edges
$E_N$ are obtained from $E$ as follows: in $E$ we delete all the
edges exiting a point of $U_N^c$ and we define $E_N$ from the
remaining edges by contraction of $U_N^c$ to the single vertex
$\delta$. By the bounded degree property we see that $G_N$ is a
finite graph. We keep the same weights on the edges and we see
that on the graph $G_N$ we have for $x$ in $U_N$
$$
\dive^N(\alpha^\gamma)(x)=\sum_{e\in E, \ue=x}
\alpha^\gamma_e-\sum_{e\in E,\oe=x, \ue\in U_N} \alpha^\gamma_e\ge
\dive(\alpha^\gamma)(x)
$$
(With $\dive^N$ the divergence operator on the graph $G_N$). Hence
$\dive^N(\alpha^\gamma)\ge 0$, and
$$
\dive^N(\alpha^\gamma)(x_0)\ge \gamma.
$$
Denote by $G^{(p)}_N(x_0,x_0)$ the Green function of the Markov
chain killed when it exits $U_N$. From corollary \ref{W} (ii) we see
that under $\P^{(\alpha^\gamma)}$ the Green function
$G_N^{(p)}(x_0,x_0)$ is stochastically dominated by ${1\over W}$
where $W$ is a beta random variable with parameters $(\gamma,
\alpha_{x_0} +\gamma\theta_{x_0} -\gamma)$ (for $N$ large enough).

Considering formula (\ref{measure-Delta}) we see that on the graph $G_N$,
the measure $\P^{(\alpha)}$ is absolutely continuous with respect to the measure
$\P^{(\alpha^\gamma)}$ and that
$$
{d\P^{(\alpha)}\over d\P^{(\alpha^\gamma)}}=
 {\prod_{x\in U_N}
\Gamma(\alpha_x)\over \prod_{e\in E_N}\Gamma
(\alpha_e)}{\prod_{e\in E_N}\Gamma (\alpha^\gamma_e)\over
\prod_{x\in U_N} \Gamma(\alpha^\gamma_x)}
\prod_{e\in E_N}
p_e^{-\gamma\theta_e}.
$$
Consider now $s>0$.  We have (we write simply $G_N(x_0,x_0)$ for
$G_N^{(p)}(x_0,x_0)$ the Green function in environment $(p_e)$)
$$
\E^{(\alpha)}\left( G_N(x_0,x_0)^s\right)= {\prod_{x\in U_N}
\Gamma(\alpha_x)\over \prod_{e\in E_N}\Gamma
(\alpha_e)}{\prod_{e\in E_N}\Gamma (\alpha^\gamma_e)\over
\prod_{x\in U_N} \Gamma(\alpha^\gamma_x)}
\E^{(\alpha^\gamma)}\left( G_N(x_0,x_0)^s \prod_{e\in E_N}
p_e^{-\gamma\theta_e}\right).
$$
Using Hölder's inequality for $q>1$ and $p={q\over q-1}$ we get
that $\E^{(\alpha)}\left( G_N(x_0,x_0)^s\right)$ is lower than
\begin{eqnarray*}
{\prod_{x\in U_N} \Gamma(\alpha_x)\over \prod_{e\in E_N}\Gamma
(\alpha_e)}{\prod_{e\in E_N}\Gamma (\alpha^\gamma_e)\over
\prod_{x\in U_N} \Gamma(\alpha^\gamma_x)} \left(
\E^{(\alpha^\gamma)}\left( G_N(x_0,x_0)^{ps}\right)\right)^{1/p}
\left(\E^{(\alpha^\gamma)}\left(\prod_{e\in E_N}
p_e^{-q\gamma\theta_e}\right)\right)^{1/q}.
\end{eqnarray*}
Remark that the second expectation is finite if and only if $
q\gamma\theta_e<\alpha^\gamma_e $ for all $e$ in $E_N$, or
equivalently
$$
q-1<{\alpha_e\over \gamma \theta_e}
$$
 for all $e$ in $E_N$.
We now choose $q$ such that
\begin{eqnarray}\label{condq}
q-1<\inf_{e\in E}{\alpha_e\over \gamma\theta_e}
\end{eqnarray}
so that the previous condition is fulfilled for all $e$ in $E$. In
terms of $p$ it is equivalent to
\begin{eqnarray}\label{condp}
p>{\inf_{e\in E}{\alpha_e\over \gamma\theta_e}+1\over \inf_{e\in
E}{\alpha_e\over \gamma\theta_e}}.
\end{eqnarray}
Now we compute the second expectation. We have
\begin{eqnarray*}
&&{\prod_{x\in U_N} \Gamma(\alpha_x)\over \prod_{e\in E_N}\Gamma
(\alpha_e)}{\prod_{e\in E_N}\Gamma (\alpha^\gamma_e)\over
\prod_{x\in U_N} \Gamma(\alpha^\gamma_x)}
\left(\E^{(\alpha^\gamma)}\left(\prod_{e\in E_N}
p_e^{-q\gamma\theta_e}\right)\right)^{1/q}
\\
&=&{\prod_{x\in U_N} \Gamma(\alpha_x)\over \prod_{e\in E_N}\Gamma
(\alpha_e)}{\prod_{e\in E_N}\Gamma (\alpha^\gamma_e)^{1-{1\over
q}}\over \prod_{x\in U_N} \Gamma(\alpha^\gamma_x)^{1-{1\over
q}}}{\prod_{e\in E_N}\Gamma
(\alpha_e-(q-1)\gamma\theta_e)^{{1\over q}}\over \prod_{x\in U_N}
\Gamma(\alpha_x-(q-1)\gamma\theta_x)^{{1\over q}}}
\end{eqnarray*}
Considering the function
$$
\nu(\alpha, u)= {1\over q} \ln\Gamma(\alpha-(q-1)u) +(1-{1\over
q})\ln\Gamma(\alpha+u) -\ln \Gamma(\alpha),
$$
we see that the previous expression is equal to
$$
\exp\left( \sum_{e\in E_N} \nu(\alpha_e,
\gamma\theta_e)-\sum_{x\in U_N} \nu(\alpha_x,
\gamma\theta_x)\right).
$$
Now, $\nu(\alpha,0)=0$ and one can easily compute and see that
$ {\partial \over \partial u}\nu(\alpha, 0)=0$. The function
$\nu(\alpha,u)$ is $C^\infty$ on the domain $D=\{\alpha>0\}\cap
\{u<\alpha/(q-1)\}$.  By conditions (H1) of theorem \ref{integrability} and (\ref{condq}) we
know that $(\alpha_e, \gamma \theta_e)_{e\in E}$ and $(\alpha_x,
\gamma\theta_x)_{x\in V}$ are in a compact subset of $D$. Hence,
we can find a constant $C>0$ such that for all $N>0$
$$
\sum_{e\in E_N} \nu(\alpha_e, \gamma\theta_e)-\sum_{x\in U_N}
\nu(\alpha_x, \gamma\theta_x) \le C \left( \sum_{e\in E_N}
(\gamma\theta_e)^2+\sum_{x\in U_N} (\gamma\theta_x)^2\right).
$$
Since $(\theta_e)$ is square integrable and the graph $G$ has
bounded degree, $(\theta_x)$ is square integrable. Hence we have a
constant $C'>0$, such that for all $N$
$$
\E^{(\alpha)}\left( G_N(x_0,x_0)^s\right)\le  C'  \left(
\E^{(\alpha^\gamma)}\left( G_N(x_0,x_0)^{ps}\right)\right)^{1/p}
$$
Using corollary \ref{W} we have
$$
\E^{(\alpha)}\left( G_N(x_0,x_0)^s\right)\le   C' \left( \E\left(
W^{-ps}\right)\right)^{1/p}
$$
where $W$ is a beta random variable with parameter $(\gamma,
\alpha_{x_0}+\gamma \theta_{x_0}-\gamma)$. Since $W^{-ps}$ is
integrable for $ps<\gamma$ we see that
$$
\E^{(\alpha)}\left( G(x_0,x_0)^s\right) =\sup_N
\E^{(\alpha)}\left( G_N(x_0,x_0)^s\right) <\infty
$$
for all $s$ such that $s p <\gamma$. This is true for any choice
of $p$ which satisfies (\ref{condp}), so $G(x_0,x_0)^s$ is
integrable as soon as
$$
s<{\inf_{e\in E}{\alpha_e\over \theta_e}\over 1+ \inf_{e\in
E}{\alpha_e\over \gamma \theta_e}}.
$$
Letting $\gamma$ tend to infinity we get the result.
\end{proof}
Maximizing on the $L_2$ unit-flows, we see that under the
conditions of theorem \ref{integrability},
$\E^{(\alpha)}\left(G(x_0,x_0)^s\right)$ is finite for all
$s<\kappa_0$ where
$$
\kappa_0=\sup_{\theta\; L_2\hbox{\small -unit flow}} \inf_E
{\alpha_e\over \theta_e}.
$$
Remark that if $\theta$ is a unit flow with finite energy then
$\inf_E {\alpha_e\over \theta_e}>0$ if condition (H1) is
satisfied. Hence, under conditions (H1) and (H2),  the existence
of a unit flow with finite energy ensures that $G(x_0,x_0)$ has
some finite moments and hence the transience of the RWDE. We see
that $\kappa_0$ can be rewritten as a Max-Flow problem (cf section
7) with an extra $L_2$ condition.
 \begin{eqnarray}\label{kappa0}
\\
\nonumber \kappa_0=\sup\{\hbox{strength}(\theta), \hbox{ $\theta$
is a flow from $x_0$ to $\infty$ of finite energy such that
$\theta\le \alpha$}\}
\end{eqnarray}

\section{Proof of transience on symmetric transient graphs}
This section is devoted to the proof of theorem \ref{t-oriented}
(i). Let $G$ be a symmetric graph and $\overline G$ the associated
undirected graph. Let us recall the definition of a flow on an
undirected graph. We choose an arbitrarily orientation of the
edges of $\overline G$. A flow from $x_0$ to infinity is a
(non-necessarily positive) function $\overline\theta$ on the edges
such that for the orientation chosen
$$
\dive(\overline\theta)(x)=0, \;\forall x\neq x_0.
$$
The flow $\overline \theta$ is a unit flow if moreover
$\dive(\overline\theta)(x_0)=1$. To any flow $\overline \theta$ on
the undirected graph $\overline G$ we can associate a flow
$\theta$ on the directed graph as follows: for two opposite edges
of the directed graph, $\theta$ is null on one of them and on the
other one it is equal to the absolute value of $\overline \theta$
on the corresponding undirected edge. The choice of the edge with
positive flow is of course made according to the sign of the flow
$\overline \theta$ and the orientation of the edges (cf \cite{PL},
section 2.6). By construction the $L_2$ norm of $\theta$ and
$\overline \theta$ are the same. Then the result comes from a
classical result on electrical networks (cf \cite{PL} proposition
2.10, or \cite{TLyons}) which says that the undirected graph
$\overline G$ is transient if and only if there exists a unit flow
with finite energy from a point $x_0$ to infinity. Hence, it
implies that $\kappa_0>0$.

\section{Max-flow of finite energy}

Let us recall some notions about Max-Flow  Min-Cut theorem (cf
\cite{PL}, section 2.6, \cite{Aharoni}). Let $G$ be a countable
directed graph and $x_0$ a vertex such that there is an infinite
directed simple path starting at $x_0$. Let $(c(e))_{e\in E}$ be a
family of non-negative reals, called the capacities.
\begin{definition}\label{flow}
A flow $\theta$ from $x_0$ to $\infty$ is
compatible with the capacities $(c(e))_{e\in E}$ if
$$
\theta(e)\le c(e), \;\;\; \forall e\in E.
$$
A cutset is a subset $S\subset E$ such that any infinite directed
simple path from $x_0$ contains at least one edge of $S$.
\end{definition}
The well-known Maw-Flow Min-Cut theorem says that the maximum flow
equals the minimal cutset sum (cf \cite{Ford}). We give here a
version for countable graphs (\cite{PL}, theorem 2.19, cf also
\cite{Aharoni}).
\begin{proposition}\label{maxflow}
The maximum compatible flow equals the infimum of the cutset sum,
i.e.
\begin{eqnarray}
\nonumber
&&\max\{\hbox{strength}(\theta), \; \hbox{ $\theta$ is a flow from
$x_0$ to $\infty$ compatible with $(c(e))$}\}
\\
\label{Mincut}
&=&\inf\{c(S),\;\hbox{ $S$ is a cutset separating $x_0$ from
$\infty$}\}.
\end{eqnarray}
where
$$
c(S)=\sum_{e\in S} c(e).
$$
\end{proposition}

Theorem \ref{integrability} tells us that $\kappa_0$, defined as
the max strength of flows of finite energy (cf (\ref{kappa0})),
gives a lower bound on the critical integrability exponent of the
Green function. It is natural to ask wether $\kappa_0$ is also
equal to the min-cut. It is not true in general (cf the following
remark) but it is true under fairly general conditions.

\begin{proposition}\label{maxflowL2}
Let $(c(e))_{e\in E}$ be a family of capacities. Suppose that
 $$\inf_{e\in E} c(e) >0,$$
and that the following holds

(H3) There exists a strictly increasing sequence of integers
$\eta_n$ such that $B(x_0, \eta_{n+1})\setminus B(x_0, \eta_n)$ is
strongly connected in $G$.

If there exists a unit flow of finite energy on $G$ from $x_0$ to
$\infty$ then the infimum in (\ref{Mincut}) is reached and
\begin{eqnarray*}
&&\max\{\hbox{strength}(\theta), \; \hbox{ $\theta$ is a flow from
$x_0$ of finite energy compatible with $(c(e))$}\}
\\
&=& \min\{c(S),\;\hbox{ $S$ is a cutset separating $x_0$ from
$\infty$}\}.
\end{eqnarray*}
\end{proposition}
\begin{remark}: If condition (H3) fails the equality may be wrong. The
following counter-example is due to R. Aharoni,
\cite{Aha-private}: let $G$ be a binary tree glued by its root to
a copy of $\Z_+$. Take capacities constant equal to 1. Then any
flow of finite energy necessarily vanishes on the copy of $\Z_+$
and the equality cannot hold.
\end{remark}
\begin{proof}
Let $\theta$ be a unit flow from $x_0$ to $\infty$ of finite
energy. Set
$$
c(G)=\inf\{c(S),\;\hbox{ $S$ is a cutset separating $x_0$ from
$\infty$}\}.
$$
The strategy is to modify the capacities $c$ using $\theta$. For a
positive integer $r$, $B_E(x_0,r)$ denotes the set of edges
$$
B_E(x_0,r)=\{e\in E, \;\; \ue\in B(x_0,r), \oe \in B(x_0,r)\}.
$$
and
$$
\underline B_E(x_0,r)=\{e\in E, \;\; \ue\in B(x_0,r) \}.
$$
 Let
$N_0$ be such that
\begin{eqnarray}\label{c}
\sup_{e\notin B_E(x_0,\eta_{N_0})} \theta(e)\le {\inf_E c\over
2c(G)}.
\end{eqnarray}
Let $N_1$ be such that
$$
(N_1-N_0)\inf_E c >  2 c(G).
$$
Consider now the capacities $c'$ defined by
$$
c'(e)=\left\{\begin{array}{ll} c(e), \; &\hbox{ if $e\in
\underline B_E(x_0,\eta_{N_1})$}\\
2 c(G) \theta(e), \; &\hbox{ if $e\notin \underline
B_E(x_0,\eta_{N_1})$}.
\end{array}
\right.
$$
By condition (\ref{c}), we clearly have
$$
c'(e) \le c(e), \forall e\in E,
$$
and
$$
\sum_E (c'(e))^2<\infty.
$$
We now want to check that the min cutset sum is the same for $c$
and $c'$. Let $S$ be a minimal cutset for inclusion. If $S\subset
\underline B_E(x_0, \eta_{N_1})$ then $c'(S)=c(S)\ge c(G)$. If
$S\subset B_E(x_0,\eta_{N_0})^c$ then by (\ref{c}) $c'(S)\ge 2c(G)
\theta(S)\ge 2c(G)$ since $\theta$ is a unit flow. Otherwise, it
means that $S$ has one edge $e_0$ in $B_E(x_0,\eta_{N_0})$ and one
edge $e_1$ in $\underline B_E(x_0, \eta_{N_1})^c$. Let $K$ be the
set of vertices that can be reached by a directed path in
$E\setminus S$ from $x_0$. Since $S$ is minimal for inclusion, it
means that there is a simple path from $\oe_0$ to $\infty$ in
$K^c$. Hence, there is a directed path in $K^c$ from $B(x_0,N_0)$
to $\infty$, and there is a sequence $y_1, \ldots, y_{N_1-N_0}$ in
$K^c$ such that $y_k\in U_k$ where
$$
U_k=B(x_0, \eta_{k+1})\setminus B(x_0, \eta_{k}).
$$
Similarly, there is a directed path in $K$ from $x_0$ to $\ue_1$.
It implies that there is a sequence $z_1, \ldots, z_{N_1-N_0}$ in
$K$ such that $z_k\in U_k$. By assumption (H3) there is a directed
path in $U_k$ from $z_k$ to $y_k$. This directed path necessarily
contains an edge of $S$. Hence, $ \vert S\vert\ge N_1-N_0$ so that
$c'(S)> 2 c(G)$. Then, we apply the Max-Flow Min-Cut theorem to
the capacities $c'$. It gives a flow of finite energy (since $c'$
is squared integrable) compatible with $c'$ and consequently with
$c$, and with strength $c(G)$. Moreover the proof implies that
$c(G)$ is reached for a cut-set $S\subset \underline B_E(x_0,
\eta_{N_1})$, and the infimum is a minimum.
\end{proof}

\section{Proof of theorem \ref{t-oriented} (ii)}

Let $\kappa_0$ be defined as in section 5 (\ref{kappa0}), and
$\kappa$ as in theorem \ref{t-oriented} (ii). Clearly, if the non-directed
graph $\overline G$ satisfies (H'3) then the directed graph $G$ satisfies
(H3). Let us first prove
that $G^s(x_0,x_0)$ is integrable if $s<\kappa$. Under condition
(H3) we have $\kappa_0\le \kappa$. Indeed, if $K$ is a finite
connected subset containing $x_0$ then $\partial_E(K)$ is a
cut-set separating $x_0$ from infinity and
$\alpha(\partial_E(K))\ge \kappa_0$ by proposition
\ref{maxflowL2}.

{\bf Case 1.} If $(x_0,x_0)$ is in $E$ then we have
$\kappa_0=\kappa$. Indeed, let $S_0$ be a cut-set which realizes
the minimum in proposition \ref{maxflowL2}, so that
$\kappa_0=\alpha(S_0)$. Let $K_0\subset V$ be the set of vertices
that can be reached from $x_0$ by a directed path using edges of
$E\setminus S_0$. By minimality we have $S_0=\partial_E(K_0)$ and
$\kappa_0=\alpha(\partial_E K_0)$. The subset $K_0$ is finite and
connected and it contains $x_0$; it implies that $\kappa=
\kappa_0$. It implies by theorem \ref{integrability} that
$\E^{(\alpha)}(G(x_0,x_0)^s)<\infty $ for $s<\kappa$.

{\bf Case 2.} Suppose that $(x_0,x_0)\not\in E$. For any
environment $(p)$, we have
$$
\sum_{\ue =x_0} p_e=1,
$$
hence $G(x_0,x_0)^s$ is integrable if and only if $p_e^s
G(x_0,x_0)^s$ is integrable for any $e$ such that $\ue=x_0$. Let
$e_0$ be an edge exiting $x_0$ and $\alpha^{(e_0)} $ be the weight
obtained from $\alpha$ by adding $s$ to $\alpha_{e_0}$. For any
cutset $S$ containing $e_0$, $\alpha^{(e_0)}(S)>s$. Hence, by
theorem \ref{integrability} and proposition \ref{maxflowL2},
$p_{e_0}^s G(x_0,x_0)^s$ is integrable if $s$ is smaller than the
minimal cutset sum (for $(\alpha)$) among cutsets which do not
contain $e_0$. It means that $G(x_0,x_0)^s$ is integrable if
\begin{eqnarray}\label{fin}
s<\min\{\alpha(S), \; \hbox{$S$ cutset, $S\not\supset
\{e\}_{\ue=x_0}$}\}.
\end{eqnarray}
Let $S$ be a cutset which does not contain $\{e\}_{\ue=x_0}$. Let
$K$ be the set of vertices that can be reached from $x_0$ by a
directed path using edges of $E\setminus S$. Then, $\partial_E K$
is a cutset contained in $S$ and $K$ is connected in $\overline
G$. Moreover $K\neq\{x_0\}$ thanks to the condition that
$(x_0,x_0)\not\in E$ (indeed, if $K=\{x_0\}$ then $\partial_E
(K)=\{e\}_{\ue =x_0}$ if $(x_0,x_0)\not\in E$). This implies that
$G^s(x_0,x_0)$ is integrable if $s<\kappa$.

If $s\ge \kappa$, then by taking a large enough box $B(x_0,N)$, we
know that the cutset which achieves the minimum in the definition
of $\kappa$ is included in $B(x_0,N)$ (the infimum is reached, cf
proposition \ref{maxflowL2}). From \cite{Tournier}, corollary 4,
it implies that $G(x_0,x_0)^s$ is not integrable (the statement is
written for $(x_0,x_0)\not\in E$ but it clearly extends to the
case where there is a loop at $x_0$ with the $\kappa$ defined in
theorem \ref{t-oriented}).

\subsection{The case of $\Z^d$}
Of course theorem \ref{t-oriented} (i) and (ii) apply to the model
of RWDE on $\Z^d$ described in the introduction, for $d\ge 3$. It
is easy to see that on $\Z^d$, the critical exponent $\kappa$ is
obtained for $K=\{0,e_{i_0}\}$ for some $i_0$ in $\{1, \ldots,
d\}$. Hence we have
$$
\kappa=\min_{i_0=1, \ldots , d}\left(2\left(\sum_{i=1, \ldots, d,
\; i\neq i_0} \alpha_{e_i}+\alpha_{-e_i}\right)+
\alpha_{e_{i_0}}+\alpha_{-e_{i_0}}\right).
$$
This proves theorem \ref{Zd}.

\section{Appendix}

\begin{proof} (of lemma \ref{chgt-variables})
Let us first remark that $\det(Z_{|U\times U})$ does not depend on
the choice of $x_0$: indeed, the lines and columns of $Z$ have sum
0. By addition of lines and columns we can switch from $x_0$ to
$y_0$. Hence, we can freely choose $x_0=\ue_0$. Remark now that
$\hhh$ and $\hhh_{e_0}$ are the affine subspaces of $\R^E$ defined
by
$$
\hhh=\cap_{x\in U}\{\dive(z)(x)=0\}, \;\;\;
\hhh_{e_0}=\{z_{e_0}=1\}\cap_{x\in U}\{\dive(z)(x)=0\}.
$$
For simplification, we write
$$
h_x(z)=\dive(z)(x),
$$
for all $x$ in $U$. It is not easy to compute directly the
Jacobian of the change of variables, the strategy is to use
Fourier transform to make the change of variables on free
variables. We first prove the following lemma.
\begin{lemma}
\label{l.Fourier} For any function $\phi:\BR^E\rightarrow \BC$,
$C^\infty$, with compact support
 in $(\BR_+^*)^E$ we have that
\begin{eqnarray*}
\int_{\hhh_{e_0}} \phi_{|\hhh_{e_0}} \prod_{e\in B^c} d z_e=
\end{eqnarray*}
\begin{eqnarray*}
\int_{\R^U\times \R} \int_{\R^E}\phi(z) \exp\left( 2i\pi\left(
u_0(z_{e_0}-1)+\sum_{x\in U } u_x h_x(z) \right)\right)
\left(\prod_{e\in E} dz_e\right) \left( du_0 \prod_{x\in U}
du_x\right).
\end{eqnarray*}
\end{lemma}
\begin{proof} {\it of lemma \ref{l.Fourier}.}
We will use several times the following simple fact (which is a
consequence of the inverse Fourier transform). Let
$g:\R^{N+k}\rightarrow \R$ be a $C^\infty$ function with compact
support, then
\begin{eqnarray*}
\label{lemme-Fourier}
 &&\int_{\R^N} g(x_1, \ldots ,x_N, 0, \ldots
,0) dx_1\cdots dx_N
\\
&=& \int_{\R^k}\int_{\R^{N+k}} \exp\left( 2 i\pi \sum_{j=1}^k u_j
x_{N+j}\right) g(x_1, \ldots ,x_{N+k}) (dx_1 \cdots dx_{N+k})
(du_1 \cdots du_k)
\end{eqnarray*}
N.B.:  These integrals are well-defined as integrals in the
Schwartz space.
 \ali
Let us compute the Jacobian of the linear change of variables
\begin{eqnarray}\label{chgt}
\nonumber \R^E&\mapsto & \R^U\times \R^{T^c}
\\
((z_e)_{e\in E}) &\to &( (h_x(z))_{x\in U}, (z_e)_{e\in T^c}).
\end{eqnarray}
Denoting $T=\{e_{j_1}, \ldots ,e_{j_\cU}\}$ and $T^c=\{e_{i_1},
\ldots ,e_{i_{\cE-\cU}}\}$, $U=\{x_1, \ldots ,x_{\vert U\vert}\}$,
the Jacobian of the change of variable is:
$$
\vert J\vert =\vert \det\left(\begin{array}{cccccc} {\partial
h_{x_1}\over \partial z_{e_{j_1}}} & \cdots &
 {\partial h_{x_\cU}\over \partial z_{e_{j_1}}}
& & &
\\
\vdots & &\vdots &&0&
\\
 {\partial h_{x_1}\over
\partial z_{e_{j_\cU}}}& \cdots &
 {\partial h_{x_\cU}\over \partial z_{e_{j_\cU}}}
 &&&
\\
 {\partial h_{x_1}\over
\partial z_{e_{i_1}}}& \cdots &
 {\partial h_{x_\cU}\over \partial z_{e_{i_1}}}
 &&&
\\
\vdots &&\vdots &&\Id_{T^c\times T^c}&
\\
{\partial h_{x_1} \over
\partial z_{e_{i_{\cE-\cU}}}}& \cdots &
 {\partial h_{x_\cU}\over \partial z_{e_{i_{\cE-\cU}}}}
 &&&
 \end{array}
 \right)\vert.
$$
So, we get
$$
\vert J\vert= \vert \det\left(\begin{array}{ccc}
 {\partial h_{x_1}\over \partial z_{e_{j_1}}} & \cdots &
 {\partial h_{x_\cU}\over \partial z_{e_{j_1}}}
\\
 \vdots &&\vdots
\\
{\partial h_{x_1}\over \partial z_{e_{j_\cU}}} & \cdots &
 {\partial h_{x_\cU}\over \partial z_{e_{j_\cU}}}
\end{array}
\right)\vert.
$$
The previous matrix is the incidence matrix on $U$ of the spanning
tree $T$, indeed we have if $x\in U$
$$
{\partial h_{x}\over \partial z_{e}}=\left\{
\begin{array}{l}
+1 \hbox{ if $\ue=x$},
\\
-1 \hbox{ if $\oe=x$},
\\
0 \hbox{ otherwise}.
\end{array}
\right.
$$
It is well know that this determinant is equal to $\pm 1$ :
indeed, it is a special case of the Kirchhoff's matrix-tree
theorem (see \cite{matrix-tree}) where the graph is only composed
of the edges of the spanning tree $T$; in this case there is a
unique spanning tree, $T$ itself, and the incident matrix of the
graph is the matrix below). Let $\tilde\phi: \R^U\times
\R^{T^c}\mapsto\C$ denote the function defined by
$$
\phi((z_e)_{e\in E})=\tilde\phi((h_x(z))_{x\in U}, (z_e)_{e\in
T^c}).
$$
By formula (\ref{lemme-Fourier}) and since
$\tilde\phi(0,(z_e)_{e\in T^c})=\phi((z_e)_{e\in E})$ on $\hhh$,
we get
\begin{eqnarray*}
\int_{\hhh_{e_0}} \phi_{|\hhh_{e_0}}((z_e)) \prod_{e\in B^c} dz_e
&=& \int_\R \int_\hhh  \phi_{|\hhh} ((z_e)) e^{2i\pi
u_0(z_{e_0-1})} \left(\prod_{e\in T^c} dz_e\right)du_0
\\
&=& \int_\R \int_\hhh  \tilde\phi(0,(z_e)_{e\in T^c}) e^{2i\pi
u_0(z_{e_0-1})} \left(\prod_{e\in T^c} dz_e\right)du_0
\end{eqnarray*}
Using again formula \ref{lemme-Fourier} we see that the previous
integral is equal to
\begin{eqnarray*}
= \int_{\R\times \BR^U} \int_{\BR^U\times \R^{T^c}}&
\exp\left(2i\pi \left( u_0(z_{e_0}-1)+\sum_{x\in U}u_x
h_x\right)\right) \tilde \phi((h_x)_{x\in U},(z_e)_{e\in T^c})
\\
& \left( \prod_{x\in U} dh_x \prod_{e\in T^c} dz_e\right)
\left(du_0 \prod_{x\in U} du_x\right).
\end{eqnarray*}
Then, the change of variables (\ref{chgt}) gives
lemma \ref{l.Fourier}.
\end{proof}

We make the following change of variables:
\begin{eqnarray*} (\BR_+^*)^E&\rightarrow&
(\R_+^*)^V\times \Delta
\\
(z_e)&\mapsto & ((v_x)_{x\in V}, (p_e)_{e\in E}),
\end{eqnarray*}
given by
$$
v_x=\sum_{\ue =x} z_e, \;\; p_e={z_e\over v_\ue}.
$$
(More precisely, it is the change of variables onto the set
$(\R_+^*)^V\times ]0,1]^{\tilde E}$ where $\tilde E$ is defined in
(\ref{measure-Delta}). It means that we choose $(p_e)_{e\in \tilde
E}$ as the coordinate system on $\Delta$.)  We have $z_e=v_\ue\;
p_e$ and the Jacobian is given by
$$
\prod_{x\in V} v_x^{n_x-1},
$$
where $n_x=\vert \{e, \ue =x\}\vert$. Implementing this change of
variables in lemma \ref{l.Fourier} gives that
\begin{eqnarray*}
&&\int_{\hhh_{e_0}} \phi_{|\hhh_{e_0}} \prod_{e\in B^c} dz_e
\end{eqnarray*}
is equal to
\begin{eqnarray*}
\int_{\R\times \BR^U} \int_{\BR^V\times \Delta }\left(\prod_{x\in
V} v_x^{n_x-1}\right)\phi((v_\ue p_e)) & \exp\left(
iu_0(v_{\ue_0}p_{e_0}-1)+ i\sum_{x\in U}u_x h_x \right)
 \\
&\left((\prod_{x\in V} dv_x) d\lambda_{\Delta}\right) \left( du_0
\prod_{x\in U} du_x\right),
\end{eqnarray*}
where $d\lambda_{\Delta}$ is the measure on  $\Delta$ defined in
(\ref{measure-Delta}) and with
$$
h_x((v_x),(p_e))= v_x-\sum_{e,\oe=x}v_{\ue}p_e.
$$
The strategy is now to apply formula (3) to get ride of variables
$(u_0,(u_x)_{x\in U})$. For this we need to change to variables
$(v_{e_0}p_{e_0}, (h_x)_{x\in U})$. We make the following change
of variables.
\begin{eqnarray*}
\R^V\times \Delta&\mapsto & \R\times\R^U\times \Delta
 \\
((v_x)_{x\in V}, (p_e))& \to &(k_0(v,p),(h_x(v,p))_{x\in U},
(p_e)),
\end{eqnarray*}
with $k_{0}=v_{\ue_0}p_{e_0}$. This change of variables can be
inverted by
$$
(v_x)_{x\in V}= (M^{(p)})^{-1}(k_0,(h_x)),
$$
where
$$
M^{(p)}_{x,y}= (I-P)_{x,y}, \;\;\; \forall x\in U, \forall y\in V,
$$
and $M_{0, y}= p_{e_0}\indic_{y=\ue_0}$ (where $P$ is the
transition matrix in the environment $(p)$). Since we have chosen
$x_0=\ue_0$, the Jacobian of the change of variables is
$$
\vert J\vert = p_{e_0} \det(I-P)_{U\times U}.
$$
By this change of variables, the integral becomes
\begin{eqnarray*}
\int_{\R\times \BR^U} \int_{\R\times\BR^U\times \Delta
}\phi((v_\ue p_e)){\left(\prod_{x\in V} v_x^{n_x-1}\right)\over
p_{e_0}\det(I-P)_{U\times U}} & \exp\left( 2i\pi\left( u_0(k_0-1)+
\sum_{x\in U}u_x h_x \right)\right)
 \\
&\left((d k_0\prod_{x\in U} dh_x) d\lambda_{\Delta}\right) \left(
du_0 \prod_{x\in U} du_x\right),
\end{eqnarray*}
Remark that if the following equalities are satisfied
$$
h_x((v_x),(p_e))=0 , \;\; \forall x\in U, \;\; v_{\ue_0}p_{e_0}=1
$$
it implies that $v_x={\pi_x\over \pi_{\ue_0}p_{e_0}}$ for all $x$.
Integrating over the variables $(u_0,(u_x))$ by formula
(\ref{lemme-Fourier}) we get
\begin{eqnarray*}
\int_{\hhh_{e_0}} \phi_{|\hhh_{e_0}} \prod_{e\in B^c} dz_e =
\int_\Delta \prod_{x\in V} \left({\pi_x\over
\pi_{\ue_0}p_{e_0}}\right)^{n_x-1} {\phi((z_e)) \over
p_{e_0}\det(I-P)_{U\times U}} d\lambda_\W,
\end{eqnarray*}
where $z_e={\pi_{\ue} p_e\over \pi_{\ue_0}p_{e_0}}$. Then we just
have to replace $\phi$ by
$$\phi((z_e))= \psi\left(({z_e\over z_{\ue}})\right) {\prod_{e\in E} \indic_{z_e>0}
z_e^{\alpha_e-1}\over\prod_{x\in V } z_x^{\alpha_x}}
\det(Z_{|U\times U})
$$
This can be done by monotone convergence.
\end{proof}

\ali \noindent{\it Acknowledgement :} I am very grateful to
Nathanaël Enriquez for several discussions on the topics of this
paper. I also thank Ron Aharoni and Russell Lyons for useful
discussions on Max-Flow Min-Cut theorem.


\begin{thebibliography}{100}
\bibitem{Aha-private}
Aharoni, R., Private communication.
\bibitem{Aharoni}
Aharoni, R.; Berger, E.; Georgakopoulos, A.; Perlstein, A.;
Sprüssel, P., The Max-Flow Min-Cut theorem for countable networks,
preprint.
 \bibitem{BZ}
Bolthausen, E.;  Zeitouni, O., Multiscale analysis of exit
distributions for random walks in random environments, Probability
Theory and Related Fields, Volume 138, Numbers 3-4 / July 2007.
\bibitem{Bremond}
Brémond, J., Recurrence of a symmetric random walk on $Z^2$ in
random medium. Juillet 2000, 12 p., Cahiers du séminaire de
probabilités de Rennes, 2002.
\bibitem{Bricmond-Kupianen}
Bricmont, J.;
 Kupiainen, A.,
Random walks in asymmetric random environments.
Comm. Math. Phys.  142  (1991),  no. 2, 345--420.
\bibitem{Conze}
Conze, J.-P.,Sur un critère de récurrence en dimension 2 pour les
marches stationnaires, applications, Ergodic Theory and Dynamical
Systems (1999), v. 19, p. 1233-1245.
\bibitem{Coppersmith-Diaconis}
Coppersmith, D., Diaconis, P., Random walks with reinforcement,
unpublished manuscript.
\bibitem{Diaconis}
Diaconis, P.,
Recent Progress in de Finetti's Notions of Exchangeability,
Bayesian Statistics,  3 J. Bernardo, et al (eds.), Oxford Press, Oxford, 111-125.
\bibitem{Diaconis-Rolles}
Diaconis, P.; Rolles, S.,
Bayesian Analysis for Reversible Markov Chains.
Annals of Statistics 34(3):1270-1292 (2006)
\bibitem{Deuschel}
Deuschel, J.-D.; Kösters, H., The quenched invariance principle for random walks in random environments
admitting a bounded cycle representation.
Ann. Inst. Henri Poincaré Probab. Stat.  44  (2008),  no. 3, 574--591.
\bibitem{Letac} Chamayou, J.-F.;  Letac, G.,
    Explicit stationary distributions for compositions of random functions
     and products of random matrices.
    J. Theoret. Probab., 4, 3--36.
\bibitem{ES2}
Enriquez, N.; Sabot, C.,  Random Walks in a Dirichlet Environment,
Electron. J. Probab.  11  (2006), no. 31, 802--817 (electronic).
\bibitem{ES1}
Enriquez, N.; Sabot, C.,  Edge oriented reinforced random walks
and RWRE.  C. R. Math. Acad. Sci. Paris 335  (2002), no. 11,
941--946.
\bibitem{ESZ1}
Enriquez, N.; Sabot, C.; Zindy, O., Limit laws for transient
random walks in random environment on $\z$, Annales de l'Institut
Fourier, Vol. 59 no. 6 (2009), p. 2469-2508.
\bibitem{ESZ2}
Enriquez, N.; Sabot, C.; Zindy, O.,   A probabilistic
representation of constants in Kesten's renewal theorem,
Probability Theory and Related Fields, Volume 144, Numbers 3-4 /
juillet 2009.
\bibitem{Ford}
Ford, L. R., Jr.; Fulkerson, D. R., Flows in networks. Princeton
University Press, Princeton, N.J. 1962 xii+194 pp.
\bibitem{matrix-tree}
Harris, J. M.; Hirst, J. L.; Mossinghoff, M. J.,  Combinatorics
and Graph Theory (Undergraduate Texts in Mathematics), Springer;
2nd ed. edition (September 19, 2008).
\bibitem{Kalikow}
Kalikow, S. A., Generalized random walk in a random environment.
Ann. Probab.  9  (1981), no. 5, 753--768.
\bibitem{Keane-Rolles}
Keane, M.S., Rolles, S., Edge-reinforced random walk on finite
graphs. In P. Clement, F. den Hollander, J. van Neerven, and B. de
Pagter, editors, Infinite dimensional stochastic analysis, pages
217-234. Koninklijke Nederlandse Akademie van Wetenschappen, 2000.
\bibitem{KKS}
Kesten, H.; Kozlov, M. V.; Spitzer, F.
A limit law for random walk in a random environment.
Compositio Math. 30 (1975), 145--168.
\bibitem{Lawler}
Lawler, G. F., Weak convergence of a random walk in a random environment.
Comm. Math. Phys.  87  (1982/83), no. 1, 81--87.
\bibitem{PL} Lyons, R.; Peres, Y., Probabilities on trees and
network, preprint, \url{http://php.indiana.edu/~rdlyons}
\bibitem{TLyons}
Lyons, T., A simple criterion for transience of a reversible
Markov chain.  Ann. Probab.  11  (1983), no. 2, 393--402.
\bibitem{Merkl1}
Merkl, Franz; Rolles, Silke W. W., Recurrence of Edge-Reinforced
Random Walk on a two-dimensional Graph. The Annals of Probability.
Volume 37, Number 5  (2009), 1679-1714.
\bibitem{Pemantle}
Pemantle, R.,
Phase transition in reinforced random walk and RWRE on trees.
Ann. Probab. 16 (1988), no. 3, 1229-1241.
\bibitem{Pemantle1}
Pemantle, R.,  Random Processes with Reinforcement.
Ph.D. Thesis, Department of Mathematics, Massachusetts Institute of Technology (1988).
\bibitem{Pemantle-survey}
Pemantle, R. A survey of random processes with reinforcement.
Probability Surveys, volume 4, pages 1-79 (2007).
\bibitem{Volkov}
Pemantle, R.; Volkov, S., Vertex-reinforced random walk on $Z$ has
finite range.  Ann. Probab.  27  (1999),  no. 3, 1368--1388.
\bibitem{Sabot-Ann}
Sabot, C., Ballistic random walks in random environment at low
disorder.  Ann. Probab.  32  (2004),  no. 4, 2996--3023.
\bibitem{hypergeom}
Sabot, C., Markov chains in a Dirichlet environment and
hypergeometric integrals.  C. R. Math. Acad. Sci. Paris  342
(2006),  no. 1, 57--62.
\bibitem{RA-S}
Rassoul-Agha, F.; Seppäläinen, T., Almost sure functional
central limit theorem for ballistic random walk in random
environment.  Ann. Inst. Henri Poincaré Probab. Stat.  45  (2009),  no. 2, 373--420.
\bibitem{RA-S2}
Rassoul-Agha, F.; Seppäläinen, T., Process-level quenched large deviations
for random walk in random environment.
Preprint.
\bibitem{Sabot-Tournier}
Sabot, C., Tournier, L., Reversed Dirichlet environment and
directional transience of random walks in Dirichlet random
environment, à paraître aux Annales de l'IHP,
\url{http://hal.archives-ouvertes.fr/hal-00387166/fr/}.
\bibitem{sz-review}
Sznitman, A.-S., Topics in random walks in random environment.
School and Conference on Probability Theory,  203--266
(electronic), ICTP Lect. Notes, XVII, Abdus Salam Int. Cent.
Theoret. Phys., Trieste, 2004.
\bibitem{SznitmanT}
Sznitman, A.-S., An effective criterion for ballistic behavior of
random walks in random environment.  Probab. Theory Related Fields
122  (2002),  no. 4, 509--544.
\bibitem{SznitmanCLT}
Sznitman, A.-S., Slowdown estimates and central limit theorem for
random walks in random environment.   J. Eur. Math. Soc. (JEMS) 2
(2000), no. 2, 93--143.
\bibitem{SZ}
Sznitman, A.-S.; Zeitouni, O., An invariance principle for
isotropic diffusions in random environment, Inventiones
Mathematicae, 164, 3, 455-567, (2006).
\bibitem{SznitmanZ}
Sznitman, A.-S.; Zerner, M.,  A law of large numbers for random
walks in random environment. Ann. Probab. 27 (1999), no. 4,
1851--1869.
\bibitem{Tarres}
Tarrès, Pierre, Vertex-reinforced random walk on $\Bbb Z$
eventually gets stuck on five points.  Ann. Probab.  32  (2004),
no. 3B, 2650--2701.
\bibitem{Tournier}
Tournier, L., Integrability of exit times and ballisticity for
random walks in Dirichlet environment, Electronic Journal of
Probability, Vol. 14 (2009), Paper 16.
\bibitem{Varadhan}
Varadhan, S. R. S., Large deviations for random walks in a random environment.
Dedicated to the memory of Jürgen K. Moser.
Comm. Pure Appl. Math.  56  (2003),  no. 8, 1222--1245.
\bibitem{Wilks}
Wilks, Samuel S. Mathematical statistics.
A Wiley Publication in Mathematical Statistics John Wiley \& Sons,
Inc., New York-London 1962 xvi+644 pp.
\bibitem{Yilmaz1}
Yilmaz, A., On the equality of the quenched and averaged large deviation rate functions for high-dimensional ballistic
random walk in a random environment. To appear in Probab. Theory Related Fields, arXiv:0903.0410, 2009.
\bibitem{Yilmaz2}
Yilmaz, A.; Zeitouni, O., Differing averaged and quenched large deviations for random walks in random environments
in dimensions two and three. arXiv:0910.1169, 2009.
\bibitem{Zeitouni}
Zeitouni, Ofer,
Random walks in random environment.
Lectures on probability theory and statistics,  189--312,
Lecture Notes in Math., 1837, Springer, Berlin, 2004.
\end{thebibliography}
\end{document}